\documentclass[11pt]{article}%
\usepackage{amsmath}
\usepackage{amsfonts}
\usepackage{amssymb}
\usepackage{amsthm}
\usepackage{graphicx}%
\newtheorem{theorem}{Theorem}[section]
\newtheorem{corollary}[theorem]{Corollary}

\newtheorem{proposition}[theorem]{Proposition}
\newtheorem{remark}[theorem]{Remark}

\theoremstyle{definition}
\newtheorem{definition}[theorem]{Definition}

\theoremstyle{remark}

\numberwithin{equation}{section}

\begin{document}

\title{A Neumann series of Bessel functions representation for solutions of
Sturm-Liouville equations}
\author{Vladislav V. Kravchenko and Sergii M. Torba\\{\small Departamento de Matem\'{a}ticas, CINVESTAV del IPN, Unidad
Quer\'{e}taro, }\\{\small Libramiento Norponiente No. 2000, Fracc. Real de Juriquilla,
Quer\'{e}taro, Qro. C.P. 76230 MEXICO}\\{\small e-mail: vkravchenko@math.cinvestav.edu.mx,
storba@math.cinvestav.edu.mx \thanks{Research was supported by CONACYT, Mexico
via the projects 166141 and 222478.}}}
\maketitle

\begin{abstract}
A Neumann series of Bessel functions (NSBF) representation for solutions of
Sturm-Liouville equations and for their derivatives is obtained. The
representation possesses an attractive feature for applications: for all real
values of the spectral parameter $\omega$ the difference between the exact
solution and the approximate one (the truncated NSBF) depends on $N$ (the
truncation parameter) and the coefficients of the equation and does not depend on $\omega$. A similar result is valid when $\omega\in\mathbb{C}$ belongs to a strip $\left\vert
\operatorname{Im}\omega\right\vert <C$. This feature makes the NSBF
representation especially useful for applications requiring computation of
solutions for large intervals of $\omega$. Error and decay rate estimates are
obtained. An algorithm for solving initial value, boundary value or spectral
problems for the Sturm-Liouville equation is developed and illustrated on a
test problem.

\end{abstract}

\section{Introduction}

In the recent work \cite{KNT} a new representation for solutions of the
one-dimensional Schr\"{o}dinger equation
\begin{equation}
-u^{\prime\prime}+Q(x)u=\omega^{2}u\label{Intro Schr}%
\end{equation}
was obtained in terms of so-called (see, e.g., \cite{Watson}, \cite{Wilkins}
and \cite{Baricz2012}) Neumann series of Bessel functions (NSBF). The
representation possesses an attractive feature for applications: for all
$\omega\in\mathbb{R}$ the difference between the exact solution and the
approximate one (the truncated NSBF) depends on $N$ (the truncation parameter) and
$Q$ and does not depend on $\omega$. A similar result is valid when $\omega
\in\mathbb{C}$ belongs to a strip $\left\vert \operatorname{Im}\omega
\right\vert <C$. This feature makes the NSBF representation especially useful
for applications requiring computation of solutions of (\ref{Intro Schr}) for
large intervals of $\omega$. For example, as was shown in \cite{KNT}, the NSBF
representation allows one to compute hundreds or if necessary even thousands
of eigendata with a nondeteriorating for large $\omega$ and remarkable
accuracy. In \cite{KTC} the NSBF representation was extended onto perturbed
Bessel equations.

In the present work we derive an NSBF representation for solutions of the
Sturm-Liouville equation
\begin{equation}
-(p(y)v^{\prime})^{\prime}+q(y)v=\omega^{2}r(y)v.\label{Intro SL}%
\end{equation}
The coefficients are assumed to admit the application of the Liouville
transformation (see, e.g., \cite{Everitt}, \cite{Zwillinger}). The main
result consists in an NSBF representation for solutions of (\ref{Intro SL})
and for their derivatives, preserving the same attractive feature described
above. Error and decay rate estimates are obtained. An algorithm for solving
initial value, boundary value or spectral problems for (\ref{Intro SL}) is
developed and illustrated on a test problem.

Besides this Introduction the paper contains the following sections. Section \ref{Section 2}
presents some well known facts concerning the Liouville transformation
together with recent results from \cite{KMT} showing how the system of formal
powers associated with (\ref{Intro SL}) is transformed by the Liouville
transformation. In Section \ref{Section 3} we recall some relevant results from \cite{KNT}
which are used in Section \ref{Section 4} and Section \ref{Section 5} for obtaining the main results of this work, the NSBF representation for solutions of (\ref{Intro SL}) as well
as for their derivatives. In Section \ref{Sect Computation formulas} convenient for computation formulas for the coefficients of the NSBF representations are derived. In Section \ref{SectNumAlg} a computational algorithm based on the NSBF representation is formulated and
discussed. Section \ref{Sect Num Examples} presents some illustrations of its numerical performance on a test problem admitting an exact solution. In Appendix A we prove error
and decay rate estimates of the NSBF representation.

\section{Preliminaries on the Liouville transformation}\label{Section 2}

\subsection{Definition of the Liouville
transformation\label{Subsect Def Liouville transform}}

Consider the Sturm-Liouville differential equation
\begin{equation}
(p(y)v^{\prime})^{\prime}-q(y)v=-\lambda r(y)v,\quad\text{ }y\in[A,B],
\label{SL}%
\end{equation}
with $[A,B]$ being a finite interval. Let $q:\left[  A,B\right]
\rightarrow\mathbb{C}$, $p$ and $r:\left[  A,B\right]  \rightarrow\mathbb{R}$
be such that $q\in C\left[  A,B\right] $, $p,p^{\prime},r,r^{\prime}\in
AC[A,B]$ and $p(y)>0,$ $r(y)>0$ for all $y\in\left[  A,B\right]  $. Define the
mapping $l:\left[  A,B\right]  \rightarrow\left[  0,b\right]  $ by
\begin{equation}
l(y):=\int_{A}^{y}\left\{  r(s)/p(s)\right\}  ^{1/2}ds, \label{l(y)}%
\end{equation}
where $b=\int_{A}^{B}\left\{  r(s)/p(s)\right\}  ^{1/2}ds$. Denote
$\rho(y):=(p(y)r(y))^{1/4}$. Then (\ref{SL}) is related to the one-dimensional
Schr\"{o}dinger differential equation
\begin{equation}
u^{\prime\prime}-Q(x)u=-\lambda u,\qquad x\in[0,b] \label{Schr}%
\end{equation}
by the Liouville transformation of the variables $y$ and $v$ into $x$ and $u$
defined as (see, e.g., \cite{Everitt}, \cite{Zwillinger})
\[
x=l(y),\qquad u(x):=\rho(y)v(y)\quad\text{ for all }y\in[A,B]
\]
and the coefficient $Q$ is given by the relation
\begin{equation}
\begin{split}
Q(x) & =\frac{q(y)}{r(y)} + \frac{p(y)}{4r(y)}\left[ \left( \frac{p^{\prime
}(y)}{p(y)}+\frac{r^{\prime}(y)}{r(y)}\right) ^{\prime}+\frac34 \left(
\frac{p^{\prime}(y)}{p(y)}\right) ^{2} + \frac12 \frac{p^{\prime}(y)}{p(y)}
\frac{r^{\prime}(y)}{r(y)} - \frac14 \left( \frac{r^{\prime}(y)}{r(y)}\right)
^{2}\right] \\
& =\frac{q(y)}{r(y)}-\frac{\rho(y)}{r(y)}\left[ p(y)\left( \frac1{\rho
(y)}\right) ^{\prime}\right] ^{\prime}\quad\text{ for all }y\in[A,B].
\end{split}
\label{Q}
\end{equation}

The Liouville transformation can be considered as an operator $\mathbf{L}%
:C[A,B]\rightarrow C[0,b]$ acting according to the rule%
\[
u(x)=\mathbf{L}[v(y)]=\rho(l^{-1}(x))v(l^{-1}(x)).
\]

Let us introduce the following notations for the differential expressions
\[
\mathbf{B}=-\frac{d^{2}}{dx^{2}}+Q(x)\quad\text{ and }\quad\mathbf{C}%
=-\frac{1}{r(y)}\left(  \frac{d}{dy}\left(  p(y)\frac{d}{dy}\right)
-q(y)\right)  .
\]

The following proposition summarizes the main properties of the operator
$\mathbf{L}$.

\begin{proposition}
\label{Prop Liouville properties}
\begin{enumerate}
\item The inverse operator is defined by $v(y)=\mathbf{L}^{-1}[u(x)]=\frac
{1}{\rho(y)}u(l(y))$.

\item The uniform norms of the operators $\mathbf{L}$ and $\mathbf{L}^{-1}$
are $\Vert\mathbf{L}\Vert=\max_{y\in\left[  A,B\right]  }\left\vert
\rho(y)\right\vert $ and $\Vert\mathbf{L}^{-1}\Vert=\max_{y\in\left[
A,B\right]  }\left\vert 1/\rho(y)\right\vert $.

\item The operator equality
\begin{equation*}
\mathbf{BL}=\mathbf{LC} 
\end{equation*}
is valid on $C^{2}\left(  A,B\right)  \cap C\left[  A,B\right]  $.
\end{enumerate}
\end{proposition}

\subsection{Transformation of formal powers}

Let $f$ be a non-vanishing solution of the equation
\begin{equation}
f^{\prime\prime}-Q(x)f=0,\quad x\in\left[  0,b\right]  ,\label{Schr hom}
\end{equation}
such that
\[
f(0)=1.
\]
On the existence of a non-vanishing $f$ see Remark
\ref{Rem Existence of non-vanishing} below. Note that in general complex
valued solutions are considered even when the coefficient $Q$ is real valued.
Denote $h:=f^{\prime}(0)\in\mathbb{C}$.

Consider two sequences of recursive integrals (see \cite{KrCV08}, \cite{KMoT},
\cite{KrPorter2010})
\begin{equation*}
X^{(0)}\equiv1,\qquad X^{(n)}(x)=n\int_{0}^{x}X^{(n-1)}(s)\left(
f^{2}(s)\right)  ^{(-1)^{n}}\,\mathrm{d}s,\qquad n=1,2,\ldots
\end{equation*}
and
\begin{equation*}
\widetilde{X}^{(0)}\equiv1,\qquad\widetilde{X}^{(n)}(x)=n\int_{0}%
^{x}\widetilde{X}^{(n-1)}(s)\left(  f^{2}(s)\right)  ^{(-1)^{n-1}}%
\,\mathrm{d}s,\qquad n=1,2,\ldots. 
\end{equation*}

\begin{definition}
\label{Def Formal powers phik and psik}The families of functions $\left\{
\varphi_{k}\right\}  _{k=0}^{\infty}$ and $\left\{  \psi_{k}\right\}
_{k=0}^{\infty}$ constructed according to the rules
\begin{equation*}
\varphi_{k}(x)=%
\begin{cases}
f(x)X^{(k)}(x), & k\text{\ odd},\\
f(x)\widetilde{X}^{(k)}(x), & k\text{\ even}%
\end{cases}
\qquad\text{and}\qquad\psi_{k}(x)=%
\begin{cases}
\dfrac{\widetilde{X}^{(k)}(x)}{f(x)}, & k\text{\ odd,}\\
\dfrac{X^{(k)}(x)}{f(x)}, & k\text{\ even}%
\end{cases}
\end{equation*}
are called systems of formal powers associated with equation \eqref{Schr}.
\end{definition}

\begin{remark}
The formal powers arise in the spectral parameter power series (SPPS)
representation for solutions of \eqref{Schr} (see \cite{KKRosu},
\cite{KrCV08}, \cite{KMoT}, \cite{KrPorter2010}).
\end{remark}

Analogously, let us introduce a system of formal powers corresponding to
equation (\ref{SL}).

Let $g$ be a solution of the equation
\begin{equation}
(p(y)g^{\prime})^{\prime}-q(y)g=0,\quad\text{ }y\in[A,B], \label{SLhom}%
\end{equation}
such that $g(y)\neq0$ for all $y\in\left[  A,B\right]  $ (see Remark
\ref{Rem Existence of non-vanishing}). Then the following two families of
auxiliary functions are well defined
\begin{align*}
\widetilde{Y}^{(0)}(y)  &  \equiv Y^{(0)}(y)\equiv1,\\
Y^{(n)}(y)  &  =%
\begin{cases}
n\int_{A}^{y}Y^{(n-1)}(s)\frac{1}{g^{2}(s)p(s)}ds, & n\text{ odd,}\\
n\int_{A}^{y}Y^{(n-1)}(s)g^{2}(s)r(s)ds, & n\text{ even,}%
\end{cases}
\\
\widetilde{Y}^{(n)}(y)  &  =%
\begin{cases}
n\int_{A}^{y}\widetilde{Y}^{(n-1)}(s)g^{2}(s)r(s)ds, & n\text{ odd,}\\
n\int_{A}^{y}\widetilde{Y}^{(n-1)}(s)\frac{1}{g^{2}(s)p(s)}ds, & n\text{
even.}%
\end{cases}
\end{align*}

Then similarly to Definition \ref{Def Formal powers phik and psik} we define
the formal powers associated to equation (\ref{SL}).

\begin{definition}
\label{Def PhikPsik} Let $p,q,r$ satisfy conditions of Subsection
\ref{Subsect Def Liouville transform} and $g$ be a non-vanishing solution of
\eqref{SLhom}. Then the formal powers associated to equation \eqref{SL} are
defined for any $k\in\mathbb{N}\cup\{0\}$ as follows
\[
\Phi_{k}(y)=%
\begin{cases}
g(y)Y^{(k)}(y), & k\text{ odd,}\\
g(y)\widetilde{Y}^{(k)}(y), & k\text{ even,}%
\end{cases}
\quad\Psi_{k}(y)=%
\begin{cases}
\frac{1}{g(y)}Y^{(k)}(y), & k\text{ even,}\\
\frac{1}{g(y)}\widetilde{Y}^{(k)}(y), & k\text{ odd.}%
\end{cases}
\]

\end{definition}

\begin{remark}
\label{Rem Existence of non-vanishing} The existence of a non-vanishing
solution of \eqref{SLhom} for complex valued $p$ and $q$ such that $p$,
$p^{\prime}$ and $q$ are continuous on $[A,B]$ was proved in \cite[Remark
5]{KrPorter2010}
(see also \cite{Camporesi et al 2011}). Moreover, the only
reason for the requirement of the absence of zeros of the functions $f$ and
$g$ is to make sure that the formal powers be well defined. As was shown in
\cite{KrTNewSPPS} this is true even when $f$ and/or $g$ have zeros, though in
that case corresponding formulas are slightly more complicated.
\end{remark}

\begin{theorem}
[\cite{KMT}]\label{relphikpsik} Let $p,q,r$ and $g$ be functions satisfying
the conditions of Definition \ref{Def PhikPsik}, and hence $f(x)=f(l(y)):=\rho
(y)g(y)$ is a particular solution of \eqref{Schr hom}. Then the following
relations are valid%
\begin{equation*}
\rho(y)\Phi_{n}(y)=\varphi_{n}(x)\quad\text{and}\quad\frac{1}{\rho(y)}\Psi
_{n}(y)=\psi_{n}(x)\text{ }\quad\text{for all }n\in\mathbb{N}\cup\left\{
0\right\}  , 
\end{equation*}
that is
\[
\varphi_{n}(x)=\mathbf{L}\left[  \Phi_{n}(y)\right]  \quad\text{and}\quad
\psi_{n}(x)=\mathbf{L}\left[  \Psi_{n}(y)/\rho^{2}(y)\right]  .
\]
\end{theorem}

In order that the equality $f(0)=1$ be fulfilled, $g$ must be chosen so that
\begin{equation}
g(A)=\frac{1}{\rho(A)}. \label{g(y0)}
\end{equation}
We will assume this initial condition to be satisfied.

It is easy to see that for $h=f^{\prime}(0)$ one obtains
\begin{equation}
h=\sqrt{\frac{p(A)}{r(A)}}\left(  \frac{g^{\prime}(A)}{g(A)}+\frac
{\rho^{\prime}(A)}{\rho(A)}\right)  . \label{h=}%
\end{equation}

\section{Solution of the one-dimensional Schr\"{o}dinger equation}\label{Section 3}

Let $\omega=\sqrt{\lambda}$, and $c(\omega,x)$, $s(\omega,x)$ denote the
solutions of (\ref{Schr}) satisfying the following initial conditions in the
origin%
\begin{equation}
c(\omega,0)=1,\quad c^{\prime}(\omega,0)=h,\quad s(\omega,0)=0,\quad
s^{\prime}(\omega,0)=\omega\label{init cs}
\end{equation}
where $h$ is an arbitrary complex number.

\begin{theorem}
[\cite{KNT}]\label{Th Representation of solutions via Bessel} The solutions
$c(\omega,x)$ and $s(\omega,x)$ of the equation
\[
u^{\prime\prime}-Q(x)u=-\omega^{2}u,\quad x\in\left(  0,b\right)
\]
admit the following representations%
\begin{equation}
c(\omega,x)=\cos\omega x+2\sum_{n=0}^{\infty}(-1)^{n}\beta_{2n}(x)j_{2n}%
(\omega x) \label{c}
\end{equation}
and%
\begin{equation}
s(\omega,x)=\sin\omega x+2\sum_{n=0}^{\infty}(-1)^{n}\beta_{2n+1}%
(x)j_{2n+1}(\omega x) \label{s}
\end{equation}
where $j_{k}$ stands for the spherical Bessel function of order $k$, the
functions $\beta_{n}$ are defined as follows
\begin{equation}
\beta_{n}(x)=\frac{2n+1}{2}\biggl(\sum_{k=0}^{n}\frac{l_{k,n}\varphi_{k}%
(x)}{x^{k}}-1\biggr), \label{beta direct definition}
\end{equation}
where $l_{k,n}$ is the coefficient of $x^{k}$ in the Legendre polynomial of
order $n$. The series in \eqref{c} and \eqref{s} converge uniformly with
respect to $x$ on $[0,b]$ and converge uniformly with respect to $\omega$ on
any compact subset of the complex plane of the variable $\omega$. Moreover,
for the functions
\begin{equation}
c_{N}(\omega,x)=\cos\omega x+2\sum_{n=0}^{\left[  N/2\right]  }(-1)^{n}%
\beta_{2n}(x)j_{2n}(\omega x) \label{cN}
\end{equation}
and
\begin{equation}
s_{N}(\omega,x)=\sin\omega x+2\sum_{n=0}^{\left[  \left(  N-1\right)
/2\right]  }(-1)^{n}\beta_{2n+1}(x)j_{2n+1}(\omega x) \label{sN}%
\end{equation}
the following uniform estimates hold
\begin{equation}
\left\vert c(\omega,x)-c_{N}(\omega,x)\right\vert \leq\sqrt{2x}\varepsilon
_{N}(x)\qquad\text{and}\qquad\left\vert s(\omega,x)-s_{N}(\omega,x)\right\vert
\leq\sqrt{2x}\varepsilon_{N}(x) \label{estc1}
\end{equation}
for any $\omega\in\mathbb{R}$, $\omega\neq0$, and
\begin{equation}
\left\vert c(\omega,x)-c_{N}(\omega,x)\right\vert \leq\varepsilon_{N}(x)
\sqrt{\frac{\sinh(2Cx)}{C}}\quad\text{and}\quad\left\vert s(\omega
,x)-s_{N}(\omega,x)\right\vert \leq\varepsilon_{N}(x) \sqrt{\frac{\sinh
(2Cx)}{C}} \label{estc2}%
\end{equation}
for any $\omega\in\mathbb{C}$, $\omega\neq0$ belonging to the strip
$\left\vert \operatorname{Im}\omega\right\vert \leq C$, $C\geq0$. Here
$\varepsilon_{N}$ is a function satisfying $\varepsilon_{N}\to0$ as
$N\to\infty$. These estimates are slight refinements for those presented in
\cite{KNT} using the ideas from \cite{KTC}. Some estimates for $\varepsilon
_{N}(x)$ are presented in Appendix \ref{Appendix Errors}.
\end{theorem}

Inequalities (\ref{estc1}) and (\ref{estc2}) reveal an interesting feature of
the representations (\ref{c}) and (\ref{s}): the accuracy of approximation of
the exact solutions by the partial sums (\ref{cN}) and (\ref{sN}) does not
depend on $\omega$ meanwhile it belongs to a strip $\left\vert
\operatorname{Im}\omega\right\vert \leq C$ in the complex plane. This feature
was tested and confirmed numerically in \cite{KNT}. An analogous result is
valid for the derivatives of the solutions. We formulate it in the following statement.

\begin{theorem}
[\cite{KNT}]\label{Th Derivatives Schrod} The derivatives of the solutions
$c(\omega,x)$ and $s(\omega,x)$ with respect to $x$ admit the following
representations
\begin{equation}
c^{\prime}(\omega,x)=-\omega\sin\omega x+\left(  h+\frac{1}{2}\int_{0}%
^{x}Q(s)\,ds\right)  \cos\omega x+2\sum_{n=0}^{\infty}(-1)^{n}\gamma
_{2n}(x)j_{2n}(\omega x) \label{c prime (omega)}
\end{equation}
and
\begin{equation}
s^{\prime}(\omega,x)=\omega\cos\omega x+\frac{1}{2}\left(  \int_{0}%
^{x}Q(s)\,ds\right)  \sin\omega x+2\sum_{n=0}^{\infty}(-1)^{n}\gamma
_{2n+1}(x)j_{2n+1}(\omega x) \label{s prime (omega)}%
\end{equation}
where $\gamma_{n}$ are defined as follows
\begin{equation}
\gamma_{n}(x)=\frac{2n+1}{2}\left(  \sum_{k=0}^{n}\frac{l_{k,n}\left(
k\psi_{k-1}(x)+\frac{f^{\prime}(x)}{f(x)}\varphi_{k}(x)\right)  }{x^{k}}%
-\frac{n(n+1)}{2x}-\frac{1}{2}\int_{0}^{x}Q(s)\,ds-\frac{h}{2}\left(
1+(-1)^{n}\right)  \right)  . \label{gamma n}
\end{equation}
The series in \eqref{c prime (omega)} and \eqref{s prime (omega)} converge
uniformly with respect to $x$ on $[0,b]$ and converge uniformly with respect
to $\omega$ on any compact subset of the complex plane of the variable
$\omega$.

Moreover, for the approximations
\begin{equation*}
\overset{\circ}{c}_{N}(\omega,x):=-\omega\sin\omega x+\left(  h+\frac{1}%
{2}\int_{0}^{x}Q(s)\,ds\right)  \cos\omega x+2\sum_{n=0}^{\left[  N/2\right]
}(-1)^{n}\gamma_{2n}(x)j_{2n}(\omega x) 
\end{equation*}
and
\begin{equation*}
\overset{\circ}{s}_{N}(\omega,x):=\omega\cos\omega x+\frac{1}{2}\left(
\int_{0}^{x}Q(s)\,ds\right)  \sin\omega x+2\sum_{n=0}^{\left[  \left(
N-1\right)  /2\right]  }(-1)^{n}\gamma_{2n+1}(x)j_{2n+1}(\omega x)
\end{equation*}
the following inequalities are valid
\[
\left\vert c^{\prime}(\omega,x)-\overset{\circ}{c}_{N}(\omega,x)\right\vert
\leq\sqrt{2x}\varepsilon_{1,N}(x)\qquad\text{and}\qquad\left\vert s^{\prime
}(\omega,x)-\overset{\circ}{s}_{N}(\omega,x)\right\vert \leq\sqrt
{2x}\varepsilon_{1,N}(x)
\]
for any $\omega\in\mathbb{R}$, $\omega\neq0$, and
\[
\left\vert c^{\prime}(\omega,x)-\overset{\circ}{c}_{N}(\omega,x)\right\vert
\leq\varepsilon_{1,N}(x)\sqrt{\frac{\sinh(2Cx)}{C}} \quad\text{and}%
\quad\left\vert s^{\prime}(\omega,x)-\overset{\circ}{s}_{N}(\omega
,x)\right\vert \leq\varepsilon_{1,N}(x)\sqrt{\frac{\sinh(2Cx)}{C}}
\]
for any $\omega\in\mathbb{C}$, $\omega\neq0$ belonging to the strip
$\left\vert \operatorname{Im}\omega\right\vert \leq C$, $C\geq0$.
\end{theorem}

\begin{remark}
[\cite{KNT}]\label{Rem eqs beta}The functions $\beta_{n}$ also can be
constructed as solutions of the recurrent equations%
\begin{equation}
\frac{1}{x^{n}}\mathbf{B}\left[  x^{n}\beta_{n}(x)\right]  =\frac{2n+1}%
{2n-3}x^{n-1}\mathbf{B}\left[  \frac{\beta_{n-2}(x)}{x^{n-1}}\right]  ,\qquad
n\ge1, \label{receqsbeta}
\end{equation}
with the first functions given by $\beta_{-1}:=1/2$ and $\beta_{0}=(f-1)/2$
and initial conditions $\sigma_{n}(0)=\sigma_{n}^{\prime}(0)=0$ where
$\sigma_{n}(x):=x^{n}\beta_{n}(x)$.

The functions $\gamma_{n}$ can be calculated from the equalities
\begin{align}
\gamma_{-1}(x)  &  = \frac14 \int_{0}^{x} Q(s)\,ds,\nonumber\\
\gamma_{0}(x)  &  =\beta_{0}^{\prime}(x)-\frac{h}{2}-\frac{1}{4}\int_{0}%
^{x}Q(s)\,ds = \frac{f^{\prime}(x)}2-\frac{h}{2}-\frac{1}{4}\int_{0}%
^{x}Q(s)\,ds ,\nonumber\\
\gamma_{n}(x)  &  =\frac{n}{x}\beta_{n}(x)+\beta_{n}^{\prime}(x)+\frac
{2n+1}{2n-3}\left(  \gamma_{n-2}(x)-\beta_{n-2}^{\prime}(x)+\frac{n-1}{x}%
\beta_{n-2}(x)\right)  ,\quad n\ge1. \label{sequence of equations for gamma}
\end{align}
One can use $\gamma_{1}(x) =\frac{1}{x}\beta_{1}(x)+\beta_{1}^{\prime
}(x)-\frac{3}{4}\int_{0}^{x}Q(s)\,ds$ as well.
\end{remark}

\begin{remark}
\label{Rem Integration procedure}Based on the formulas from Remark
\ref{Rem eqs beta} a stable recurrent integration procedure for computing the
functions $\sigma_{n}$ and $\tau_{n}:=x^{n}\gamma_{n}$ was proposed in
\cite{KNT} in the following form.
\begin{align}
\eta_{n}(x)  &  =\int_{0}^{x}\bigl(tf^{\prime}(t)+(n-1)f(t)\bigr)\sigma
_{n-2}(t)\,dt,\quad\theta_{n}(x)=\int_{0}^{x}\frac{1}{f^{2}(t)}%
\bigl(\eta_{n}(t)-tf(t)\sigma_{n-2}(t)\bigr)dt,\label{eta n AND theta n}\\
\sigma_{n}(x)  &  =\frac{2n+1}{2n-3}\left[  x^{2}\sigma_{n-2}(x)+c_{n}%
f(x)\theta_{n}(x)\right]  ,\label{sigma n}\\
\tau_{n}(x)  &  =\frac{2n+1}{2n-3}\left[  x^{2}\tau_{n-2}(x)+c_{n}\left(
f^{\prime}(x)\theta_{n}(x)+\frac{\eta_{n}(x)}{f(x)}\right)  -(c_{n}%
-2n+1)x\sigma_{n-2}(x)\right]  ,\quad n\ge1, \label{tau n}%
\end{align}
where $c_{n}=1$ if $n=1$ and $c_{n}=2(2n-1)$ otherwise.
\end{remark}

\section{Solution of the Sturm-Liouville equation}\label{Section 4}

Let us generalize Theorems \ref{Th Representation of solutions via Bessel} and
\ref{Th Derivatives Schrod} onto solutions of equation (\ref{SL}).

\begin{theorem}
Let the functions $p$, $q$ and $r$ satisfy the conditions from Subsection
\ref{Subsect Def Liouville transform} and $g$ be a solution of \eqref{SLhom}
satisfying \eqref{g(y0)} and such that $g(y)\neq0$ for all $y\in\left[
A,B\right]  $. Then two linearly independent solutions $v_{1}$ and $v_{2}$ of
equation \eqref{SL} for $\omega\neq0$ can be written in the form
\begin{equation}
v_{1}(\omega,y)=\frac{\cos\left(  \omega l(y)\right)  }{\rho(y)}+2\sum
_{n=0}^{\infty}(-1)^{n}\alpha_{2n}(y)j_{2n}(\omega l(y)) \label{v1}
\end{equation}
and
\begin{equation}
v_{2}(\omega,y)=\frac{\sin\left(  \omega l(y)\right)  }{\rho(y)}+2\sum
_{n=0}^{\infty}(-1)^{n}\alpha_{2n+1}(y)j_{2n+1}(\omega l(y)) \label{v2}%
\end{equation}
with $l(y)$ defined by \eqref{l(y)}, the coefficients $\alpha_{n}$ being
defined by the equalities
\begin{equation}
\alpha_{n}(y)=\frac{2n+1}{2}\left(  \sum_{k=0}^{n}\frac{l_{k,n}\Phi_{k}%
(y)}{l^{k}(y)}-\frac{1}{\rho(y)}\right)  , \label{alpha_n}
\end{equation}
where $\Phi_{k}$ are from Definition \ref{Def PhikPsik}. The solutions $v_{1}$
and $v_{2}$ satisfy the following initial conditions
\begin{align}
v_{1}(\omega,A) & =\frac{1}{\rho(A)},\qquad v_{2}(\omega,A)=0,\label{init v1}
\\
v_{1}^{\prime}(\omega,A) & =g^{\prime}(A),\qquad v_{2}^{\prime}(\omega
,A)=\frac{\omega}{\rho(A)}\sqrt{\frac{r(A)}{p(A)}}. \label{init v2}
\end{align}
The series in \eqref{v1} and \eqref{v2} converge uniformly with respect to $y$
on $[A,B]$ and converge uniformly with respect to $\omega$ on any compact
subset of the complex plane of the variable $\omega$.

Moreover, for the functions
\begin{equation}
v_{1}^{N}(\omega,y)=\frac{\cos\left(  \omega l(y)\right)  }{\rho(y)}%
+2\sum_{n=0}^{\left[  N/2\right]  }(-1)^{n}\alpha_{2n}(y)j_{2n}(\omega l(y))
\label{v1N}
\end{equation}
and
\begin{equation}
v_{2}^{N}(\omega,y)=\frac{\sin\left(  \omega l(y)\right)  }{\rho(y)}%
+2\sum_{n=0}^{\left[  \left(  N-1\right)  /2\right]  }(-1)^{n}\alpha
_{2n+1}(y)j_{2n+1}(\omega l(y)) \label{v2N}%
\end{equation}
the following estimates hold%
\begin{equation}
\left\vert v_{1,2}(\omega,y)-v_{1,2}^{N}(\omega,y)\right\vert \leq\sqrt
{2l(y)}\varepsilon_{N}(l(y))\max_{y\in\left[  A,B\right]  }\frac{1}{\left\vert
\rho(y)\right\vert } \label{v12 estimate}
\end{equation}
for any $\omega\in\mathbb{R}$, $\omega\neq0$, and
\begin{equation}
\left\vert v_{1,2}(\omega,y)-v_{1,2}^{N}(\omega,y)\right\vert \leq
\varepsilon_{N}(l(y))\sqrt{\frac{\sinh(2Cl(y))}{C}}\max_{y\in\left[
A,B\right]  }\frac{1}{\left\vert \rho(y)\right\vert } \label{v12 estimate C}%
\end{equation}
for any $\omega\in\mathbb{C}$, $\omega\neq0$ belonging to the strip
$\left\vert \operatorname{Im}\omega\right\vert \leq C$, $C\geq0$, where
$\varepsilon_{N}$ is a function from Theorem
\ref{Th Representation of solutions via Bessel}.
\end{theorem}

\begin{proof}
Consider solutions (\ref{c}) and (\ref{s}) of (\ref{Schr}). The functions
\[
v_{1}(\omega,y)=\mathbf{L}^{-1}\left[  c(\omega,x)\right]  =\frac{1}{\rho
(y)}\left(  \cos\left(  \omega l(y)\right)  +2\sum_{n=0}^{\infty}(-1)^{n}%
\beta_{2n}(l(y))j_{2n}(\omega l(y))\right)
\]
and
\[
v_{2}(\omega,y)=\mathbf{L}^{-1}\left[  s(\omega,x)\right]  =\frac{1}{\rho
(y)}\left(  \sin\left(  \omega l(y)\right)  +2\sum_{n=0}^{\infty}(-1)^{n}%
\beta_{2n+1}(l(y))j_{2n+1}(\omega l(y))\right)
\]
are then solutions of (\ref{SL}) with $\omega^{2}=\lambda$. In these
representations all the magnitudes except the coefficients $\beta_{n}$ can be
defined with no reference to equation (\ref{Schr}). Let us show that the same
observation is applicable to the functions
\[
\alpha_{n}(y):=\frac{\beta_{n}(l(y))}{\rho(y)}=\mathbf{L}^{-1}\left[
\beta_{n}(x)\right]  ,\quad n=0,1,2,\ldots.
\]
Note that
\begin{equation}
\mathbf{L}^{-1}\left[  x^{k}u(x)\right]  =l^{k}(y)\frac{u(l(y))}{\rho
(y)}=l^{k}(y)\mathbf{L}^{-1}\left[  u(x)\right]  \label{propL-1}%
\end{equation}
for any function $u$ and any $k\in\mathbb{R}$. Then
\[
\alpha_{n}(y) =\frac{2n+1}{2}\mathbf{L}^{-1}\left[  \sum_{k=0}^{n}%
\frac{l_{k,n}\varphi_{k}(x)}{x^{k}}-1\right]  =\frac{2n+1}{2}\left(
\sum_{k=0}^{n}\frac{l_{k,n}\mathbf{L}^{-1}\left[  \varphi_{k}(x)\right]
}{l^{k}(y)}-\frac{1}{\rho(y)}\right)  .
\]
Due to Theorem \ref{relphikpsik} we obtain (\ref{alpha_n}) and thus the
solutions $v_{1}$ and $v_{2}$ of (\ref{SL}) have the form (\ref{v1}),
(\ref{v2}).

Estimates (\ref{v12 estimate}) and (\ref{v12 estimate C}) follow from
(\ref{estc1}) and (\ref{estc2}) by taking into account Proposition
\ref{Prop Liouville properties}.

Equalities (\ref{init v1}) and (\ref{init v2}) are obtained directly from
(\ref{init cs}). Indeed, due to statement 1. from Proposition
\ref{Prop Liouville properties}, $v_{1}(\omega,A)=\frac{c(\omega,0)}{\rho(A)}$
and $v_{2}(\omega,A)=\frac{s(\omega,0)}{\rho(A)}$ that gives us (\ref{init v1}%
). From the definition of the Liouville transformation we have the equality
$(\rho v)_{y}=u_{x}l_{y}$ for $u$ and $v$ being related by $u(x)=\mathbf{L}%
[v(y)]$, and thus,%
\begin{equation}
v_{y}=\frac{1}{\rho}\left(  \sqrt{\frac{r}{p}}u_{x}-\rho_{y}v\right)  ,
\label{der Liouville}%
\end{equation}
from where (\ref{init v2}) are derived.
\end{proof}

\begin{remark}
Inequalities \eqref{v12 estimate} and \eqref{v12 estimate C} show that the
representations \eqref{v1} and \eqref{v2} preserve the important property of
the representations \eqref{c} and \eqref{s}: the accuracy of approximation of
the exact solutions by the partial sums \eqref{v1N} and \eqref{v2N} does not
depend on $\omega$ meanwhile it belongs to a strip $\left\vert
\operatorname{Im}\omega\right\vert \leq C$ in the complex plane.
\end{remark}

\begin{remark}
Similarly to the coefficients $\beta_{n}$ (see Remark \ref{Rem eqs beta}) the
functions $\alpha_{n}$ can be constructed as solutions of the recurrent
equations%
\begin{equation}
\frac{1}{l^{n}(y)}\mathbf{C}\left[  l^{n}(y)\alpha_{n}(y)\right]  =\frac
{2n+1}{2n-3}l^{n-1}(y)\mathbf{C}\left[  \frac{\alpha_{n-2}(y)}{l^{n-1}%
(y)}\right]  \label{receqsalpha}
\end{equation}
obtained by applying $\mathbf{L}^{-1}$ to \eqref{receqsbeta} and taking into
account \eqref{propL-1} and the equality $\mathbf{L}^{-1}\mathbf{B}%
=\mathbf{CL}^{-1}$. The first functions of the sequence $\left\{  \alpha
_{n}\right\}  $ are given by
\begin{equation}
\alpha_{-1}=\mathbf{L}^{-1}\left[  \beta_{-1}\right]  =\frac{1}{2\rho}%
\qquad\text{and}\qquad\alpha_{0}=\mathbf{L}^{-1}\left[  \beta_{0}\right]
=\frac{1}{2}\left(  g-\frac{1}{\rho}\right)  . \label{alpha -1}
\end{equation}
The initial conditions satisfied by $\Sigma_{n}(y):=l^{n}(y)\alpha_{n}(y)$
have the form%
\begin{equation}
\Sigma_{n}(A)=\Sigma_{n}^{\prime}(A)=0. \label{initcondSigma}%
\end{equation}

This sequence of equations for $\left\{  \alpha_{n}\right\}  $ leads to a
stable recursive integration procedure which is proposed in Section
\ref{Sect Computation formulas}.
\end{remark}

\section{Representation of the derivatives of the solutions $v_{1}$ and
$v_{2}$}\label{Section 5}

Due to (\ref{der Liouville}) for $v_{1}(\omega,y)=\mathbf{L}^{-1}\left[
c(\omega,x)\right]  $ and $v_{2}(\omega,y)=\mathbf{L}^{-1}\left[
s(\omega,x)\right]  $ with the aid of (\ref{c prime (omega)}),
(\ref{s prime (omega)}) and (\ref{v1}), (\ref{v2}) we obtain
\begin{equation}%
\begin{split}
v_{1}^{\prime}(\omega,y)  & =\sqrt{\frac{r(y)}{p(y)}}\left(  \frac{1}{\rho
(y)}\left(  G_{1}(y)\cos(\omega l(y))-\omega\sin(\omega l(y))\right)
+2\sum_{n=0}^{\infty}(-1)^{n}\mu_{2n}(y)j_{2n}(\omega l(y))\right) \\
& \quad-\frac{\rho^{\prime}(y)}{\rho(y)}v_{1}(\omega,y)\label{v1 prime}%
\end{split}
\end{equation}
and
\begin{equation}%
\begin{split}
v_{2}^{\prime}(\omega,y)  &  =\sqrt{\frac{r(y)}{p(y)}}\left(  \frac{1}%
{\rho(y)}\left(  G_{2}(y)\sin(\omega l(y))+\omega\cos(\omega l(y))\right)
+2\sum_{n=0}^{\infty}(-1)^{n}\mu_{2n+1}(y)j_{2n+1}(\omega l(y))\right) \\
& \quad-\frac{\rho^{\prime}(y)}{\rho(y)}v_{2}(\omega,y)\label{v2 prime}%
\end{split}
\end{equation}
where
\[
\mu_{n}(y):=\frac{\gamma_{n}(l(y))}{\rho(y)}=\mathbf{L}^{-1}\left[  \gamma
_{n}(x)\right]  ,\quad n=0,1,2,\ldots,
\]
and the functions $G_{1}$ and $G_{2}$ have the form (cf. \cite[Remark
4.7]{KMT})
\begin{equation}
\label{G1(y)}
\begin{split}
G_{1}(y):= & h+\frac{1}{2}\int_{0}^{l(y)}Q(s)\,ds=h+\frac{1}{2}\int_{A}%
^{y}\frac{1}{(pr)^{1/4}}\left(  \frac{q}{(pr)^{1/4}}-[p\{(pr)^{-1/4}%
\}^{\prime}]^{\prime}\right)  (s)ds\\
= &  h + \left. \frac{\rho\rho^{\prime}}{2r} \right| _{A}^{y} +\frac12
\int_{A}^{y} \left[  \frac{q}{\rho^{2}}+ \frac{(\rho^{\prime})^{2}}r\right]
(s)\,ds
\end{split}
\end{equation}
and
\begin{equation}
G_{2}(y):=\frac{1}{2}\int_{0}^{l(y)}Q(s)\,ds=G_{1}(y)-h. \label{G2(y)}%
\end{equation}
Let us obtain a formula for calculating $\mu_{n}$ which would not involve
magnitudes related to equation (\ref{Schr}) but only those related to
(\ref{SL}). Consider (\ref{gamma n}). Direct calculation gives us the
relation
\[
\frac{f^{\prime}(x)}{f(x)}=\sqrt{\frac{p(y)}{r(y)}}\left(  \frac{g^{\prime
}(y)}{g(y)}+\frac{\rho^{\prime}(y)}{\rho(y)}\right)  .
\]
Hence, due to Theorem \ref{relphikpsik},
\[
k\psi_{k-1}(x)+f^{\prime}(x)\varphi_{k}(x)/f(x)=k\frac{\Psi_{k-1}(y)}{\rho
(y)}+\rho(y)\sqrt{\frac{p(y)}{r(y)}}\left(  \frac{g^{\prime}(y)}{g(y)}%
+\frac{\rho^{\prime}(y)}{\rho(y)}\right)  \Phi_{k}(y).
\]
Thus,
\begin{equation}%
\begin{split}
\mu_{n}(y)  &  =\frac{2n+1}{2\rho(y)}\left(  \sum_{k=0}^{n}\frac{l_{k,n}%
}{l^{k}(y)}\left(  k\frac{\Psi_{k-1}(y)}{\rho(y)}+\rho(y)\sqrt{\frac
{p(y)}{r(y)}}\left(  \frac{g^{\prime}(y)}{g(y)}+\frac{\rho^{\prime}(y)}%
{\rho(y)}\right)  \Phi_{k}(y)\right)  \right. \\
&  \quad\left.  -\frac{n(n+1)}{2l(y)}-G_{2}(y)-\frac{h}{2}\left(
1+(-1)^{n}\right)  \right)  .\label{mu_n}%
\end{split}
\end{equation}
This is a direct formula for calculating $\mu_{n}$ in terms of the functions
$\Phi_{k}$ and $\Psi_{k}$. A recurrent formula for $\mu_{n}$ can be obtained
from (\ref{sequence of equations for gamma}). We have%
\begin{align}
\mu_{-1}(y) & = \frac{G_{2}(y)}{2\rho(y)},\label{mu -1}\\
\mu_{0}(y) & =\sqrt{\frac{p(y)}{r(y)}}\left(  \alpha_{0}^{\prime}%
(y)+\frac{\rho^{\prime}(y)}{\rho(y)}\alpha_{0}(y)\right)  -\frac{G_{1}%
(y)}{2\rho(y)} = \sqrt{\frac{p(y)}{r(y)}} \frac{\bigl(g(y) \rho
(y)\bigr)^{\prime}}{2\rho(y)}-\frac{G_{1}(y)}{2\rho(y)},\label{mu 0}%
\end{align}
and%
\begin{multline}\label{mu n}
\mu_{n}(y)=\frac{n\alpha_{n}(y)}{l(y)}+\sqrt{\frac{p(y)}{r(y)}}\left(
\alpha_{n}^{\prime}(y)+\frac{\rho^{\prime}(y)}{\rho(y)}\alpha_{n}(y)\right) \\
+\frac{2n+1}{2n-3}\left(  \mu_{n-2}(y)-\sqrt{\frac{p(y)}{r(y)}}\left(
\alpha_{n-2}^{\prime}(y)+\frac{\rho^{\prime}(y)}{\rho(y)}\alpha_{n-2}%
(y)\right)  +\frac{n-1}{l(y)}\alpha_{n-2}(y)\right)  ,\quad n=1,2,\ldots.
\end{multline}
One may also use
\begin{equation}
\mu_{1}(y)=\frac{\alpha_{1}(y)}{l(y)}+\sqrt{\frac{p(y)}{r(y)}}\left(
\alpha_{1}^{\prime}(y)+\frac{\rho^{\prime}(y)}{\rho(y)}\alpha_{1}(y)\right)
-\frac{3G_{2}(y)}{2\rho(y)}.\label{mu 1}%
\end{equation}

\section{Computation formulas for the
coefficients\label{Sect Computation formulas}}

Besides the direct formulas (\ref{alpha_n}) and (\ref{mu_n}) for the
coefficients $\alpha_{n}$ and $\mu_{n}$ appearing in the representations
(\ref{v1}), (\ref{v2}), (\ref{v1 prime}) and (\ref{v2 prime}) it is convenient
to have formulas more appropriate for practical computation. The main drawback
of formulas (\ref{alpha_n}) and (\ref{mu_n}) for numerical computing is the
presence of the Legendre polynomial coefficients $l_{k,n}$ which grow rather
fast with a growing $n$. Here we derive a recurrent integration procedure for
computing $\alpha_{n}$ and $\mu_{n}$, convenient for numerical applications.

Formulas from Remark \ref{Rem Integration procedure} can be used for writing
the recurrent integration procedure aimed to compute the functions
\begin{equation}
\label{Sigma and Upsilon}\Sigma_{n}(y):=\mathbf{L}^{-1}\left[  \sigma
_{n}(x)\right]  =l^{n}(y)\alpha_{n}(y)\qquad\text{and}\qquad\Upsilon
_{n}(y):=\mathbf{L}^{-1}\left[  \tau_{n}(x)\right]  =l^{n}(y)\mu_{n}(y).
\end{equation}
From (\ref{sigma n}) we obtain%
\[
\Sigma_{n}(y)=\frac{2n+1}{2n-3}\left(  l^{2}(y)\Sigma_{n-2}(y)+c_{n}%
\mathbf{L}^{-1}\left[  f(x)\theta_{n}(x)\right]  \right) ,
\]
where $c_{n}=1$ if $n=1$ and $c_{n}=2(2n-1)$ otherwise.

Denote $\widetilde{\theta}_{n}(y):=\theta_{n}(l(y))$. Then $\mathbf{L}%
^{-1}\left[  f(x)\theta_{n}(x)\right]  =g(y)\widetilde{\theta}_{n}(y)$. Hence%
\begin{equation}
\Sigma_{n}(y)=\frac{2n+1}{2n-3}\left(  l^{2}(y)\Sigma_{n-2}(y)+c_{n}%
g(y)\widetilde{\theta}_{n}(y)\right)  . \label{Sigma n}%
\end{equation}
Analogously, for $n\in\mathbb{N}$,
\begin{equation}%
\begin{split}
\Upsilon_{n}(y)  &  =\frac{2n+1}{2n-3}\left(  l^{2}(y)\Upsilon_{n-2}%
(y)+c_{n}\left(  \sqrt{\frac{p(y)}{r(y)}}\left(  g^{\prime}(y)\rho
(y)+g(y)\rho^{\prime}(y)\right)  \frac{\widetilde{\theta}_{n}(y)}{\rho
(y)}+\frac{\widetilde{\eta}_{n}(y)}{\rho^{2}(y)g(y)}\right)  \right.
\label{Tau n}\\
&  \quad\left.  -\left(  c_{n}-2n+1\right)  l(y)\Sigma_{n-2}(y)\right)
\end{split}
\end{equation}
where $\widetilde{\eta}_{n}(y):=\eta_{n}(l(y))$.

A simple change of the integration variable $t=l(\tau)$ in
(\ref{eta n AND theta n}) gives us the equalities%
\begin{equation}
\widetilde{\eta}_{n}(y)=\int_{A}^{y}\left(  l(\tau)\left(  g^{\prime}%
(\tau)\rho(\tau)+g(\tau)\rho^{\prime}(\tau)\right)  +(n-1)\rho(\tau
)g(\tau)\sqrt{\frac{r(\tau)}{p(\tau)}}\right)  \rho(\tau)\Sigma_{n-2}%
(\tau)d\tau\label{eta tilde}%
\end{equation}
and
\begin{equation}
\widetilde{\theta}_{n}(y)=\int_{A}^{y}\left(  \frac{\widetilde{\eta}_{n}%
(\tau)}{\rho^{2}(\tau)g^{2}(\tau)}-\frac{l(\tau)\Sigma_{n-2}(\tau)}{g(\tau
)}\right)  \sqrt{\frac{r(\tau)}{p(\tau)}}d\tau. \label{theta tilde}%
\end{equation}

\section{Description of the numerical algorithm}

\label{SectNumAlg}

Numerical solution of equation (\ref{SL}) for different values of $\lambda$
can be computed following the sequence:

\begin{enumerate}
\item Find a nonvanishing solution $g$ of (\ref{SLhom}) on $[A,B]$ satisfying
the initial condition (\ref{g(y0)}). The solution $g$ can be constructed using
the SPPS representation, see, e.g., \cite{KrPorter2010} for details, or by
means of any other numerical method.

\item Compute the constant $h$ from (\ref{h=}), $\alpha_{-1}$ and $\alpha_{0}$
from (\ref{alpha -1}), $G_{1}$ and $G_{2}$ from (\ref{G1(y)}) and
(\ref{G2(y)}), $\mu_{-1}$ and $\mu_{0}$ from (\ref{mu -1}) and (\ref{mu 0}).

\item Compute the set of functions $\left\{  \Sigma_{n},\Upsilon
_{n},\widetilde{\eta}_{n},\widetilde{\theta}_{n}\right\}  _{n=1}^{N}$ which
gives us the coefficients $\alpha_{n}=\Sigma_{n}/l^{n}$ and $\mu_{n}%
=\Upsilon_{n}/l^{n}$.

\item Compute the approximations of the solutions $v_{1}$ and $v_{2}$ and of
their derivatives $v_{1}^{\prime}$ and $v_{2}^{\prime}$ using (\ref{v1}),
(\ref{v2}) and (\ref{v1 prime}), (\ref{v2 prime}) with the infinite series
replaced by corresponding partial sums.

\item Use the initial conditions (\ref{init v1}) and (\ref{init v2}) satisfied
by $v_{1}$ and $v_{2}$ to obtain a solution satisfying initial conditions
prescribed for solving an initial or boundary value problem and/or a spectral
problem related to (\ref{SL}).
\end{enumerate}

We refer the reader to \cite{KT AnalyticApprox}, \cite{KMT}, \cite{KNT} for
implementation details regarding the recurrent integration and computation of
Bessel functions.

Additionally we would like to point out that numerical computation of the
coefficients $\alpha_{n}$ and $\mu_{n}$ either via direct formulas
\eqref{alpha_n} and \eqref{mu_n} or via recurrent formulas
\eqref{Sigma and Upsilon}, \eqref{Sigma n} and \eqref{Tau n} can be tricky for
values of $y$ near the point $y=A$. One of the reasons is that in recursive
formulas an absolute error in any coefficient $\alpha_{n}$ propagates to all
the following coefficients $\alpha_{n+2k}$ and in the direct formulas absolute
errors in the first functions $\Phi_{k}$ are get multiplied by quite large
Legendre polynomials' coefficients $l_{k,n}$. Another reason is the division
by large powers of $l(y)$ in both \eqref{Sigma and Upsilon} and
\eqref{alpha_n}, \eqref{mu_n}. I.e., the coefficients $\alpha_{n}$ and
$\mu_{n}$ possess absolute errors which we can not expect to decrease as
$n\to\infty$, while the absolute values of the coefficients $\alpha_{n}$ and
$\mu_{n}$ decay as $n\to\infty$ and are bounded by decaying functions as $y\to
A$, see Proposition \ref{Prop Decay betas} and Corollary
\ref{Corr decay betas}. That is, the computed coefficients will possess large
relative errors especially near the point $y=A$.

One possibility to overcome the difficulties mentioned above is to change the
values of the computed coefficients $\alpha_{n}$ and $\mu_{n}$ near $y=A$ by
zero, i.e., we take $\alpha_{n}(y)=0$ for all $A\leq y\leq y_{n}$ and $\mu
_{n}(y)=0$ for all $A\leq y\leq\tilde{y}_{n}$. The numbers $y_{n}$ and
$\tilde{y}_{n}$ may be estimated using the following equalities (which easily
follow from \eqref{K GoursatCond}, \eqref{K1 GoursatCond} and
\eqref{K and K1 Legendre} by applying the operator $\mathbf{L}^{-1}$ and
\eqref{propL-1}).
\begin{equation}
\sum_{n=0}^{\infty}\frac{\alpha_{n}(y)}{l(y)}=\frac{G_{1}(y)+G_{2}(y)}{2\rho(y)}%
,\qquad\sum_{n=0}^{\infty}\frac{(-1)^{n}\alpha_{n}(y)}{l(y)}=\frac{h}%
{2\rho(y)}\label{alpha n verification}%
\end{equation}
and
\begin{align}
\sum_{n=0}^{\infty}\frac{\mu_{n}(y)}{l(y)} &  =\frac{q(y)}{4\rho(y)r(y)}%
-\frac{1}{4r(y)}\left[  p(y)\left(  \frac{1}{\rho(y)}\right)  ^{\prime
}\right]  ^{\prime}+\frac{hG_{2}(y)+G_{2}^{2}(y)}{2\rho(y)}%
,\label{mu n verification}\\
\sum_{n=0}^{\infty}\frac{(-1)^{n}\mu_{n}(y)}{l(y)} &  =\frac{1}{4\rho
(y)}\left(  \frac{q(A)}{r(A)}-\frac{\rho(A)}{r(A)}\left.  \left[  p(y)\left(
\frac{1}{\rho(y)}\right)  ^{\prime}\right]  ^{\prime}\right\vert
_{y=A}\right)  +\frac{hG_{2}(y)}{2\rho(y)}.\label{mu n verification2}
\end{align}
One computes differences between the right-hand sides of
\eqref{alpha n verification}, \eqref{mu n verification},
\eqref{mu n verification2} and partial sums of the series from the left-hand
sides, and finds the moment when these differences cease to decrease and reach
some floor value as $N$ (the number of terms in the partial sums) grows.
Analyzing these values of $N$ for each $y$ one may estimate the values $y_{n}$
and $\tilde{y}_{n}$.

Another simple rule for choosing $y_{n}$ and $\tilde y_{n}$ which we found
being quite acceptable numerically consists in the following. After computing
a coefficient $\alpha_{n}$ for some $n$, we take $y_{n}$ as a point in a
neighborhood of $y=A$ where $\alpha_n(y)\ne 0$ and  $|\alpha_{n}(y)|$ attains its minimum. The same
procedure for the coefficients $\mu_{n}$. As an explanation for this rule we
refer to Proposition \ref{Prop Decay betas} and Corollary
\ref{Corr decay betas} from which it follows that the coefficients $\alpha
_{n}$ and $\mu_{n}$ are bounded by some powers of $l(y)$ near $y=A$ and
$l(y)\to0$ as $y\to A$.

It is worth mentioning that as follows from the inequality \cite[(9.1.62)]%
{Abramowitz}
\[
|j_{n}(z)|\leq\sqrt{\pi}\left\vert \frac{z}{2}\right\vert ^{n}\frac
{e^{\operatorname{Im}z}}{\Gamma(n+3/2)},\qquad z\in\mathbb{C},
\]
the values of the function $j_{n}(z)$ for small $z$ and large $n$ are so small
(for example, $j_{40}(1)\approx1.5\cdot10^{-61}$, $j_{40}(10)\approx
8.4\cdot10^{-22}$) that even huge errors in the coefficients $\alpha_{n}(y)$
and $\mu_{n}(y)$ in a neighborhood of $y=A$ do not lead to significant errors
in the approximate solutions. Only for large values of $\omega$ (corresponding
to 100th eigenvalue and further) the error can be noticeable, see Section
\ref{Sect Num Examples} for an illustration.

\section{Numerical experiments}

\label{Sect Num Examples} We implemented the algorithm from Section
\ref{SectNumAlg} in Matlab 2012 in machine precision arithmetics. All the
calculations were performed as was described in \cite{KMT}, \cite{KNT},
\cite{KTC}. We opted for a straightforward implementation to illustrate that
the proposed method is capable to produce highly accurate results as is.
Clearly it can benefit even more being combined with such techniques as
interval subdivision etc.

We considered the following equation \cite{KMT}, \cite[Eqn. 2.273(11)]{Kamke}
\begin{equation}
u^{\prime\prime}-2u^{\prime}+u=-\lambda(y^{2}+1)u,\qquad y\in\lbrack
0,2]\label{Num Eq1}%
\end{equation}
together with the boundary conditions
\begin{equation}
u(0)-u^{\prime}(0)=0,\qquad u(2)+u^{\prime}(2)=0.\label{Num Eq1 BC}%
\end{equation}
Equation \eqref{Num Eq1} can be transformed into the form \eqref{SL} with
$p(y)=e^{-2y}$, $q(y)=-e^{-2y}$ and $r(y)=(y^{2}+1)e^{-2y}$.

The algorithm from our previous paper \cite{KMT} was able to compute the first
100 eigenvalues for this problem with the maximum absolute error of
$4.7\cdot10^{-7}$ and maximum relative error of $8.5\cdot10^{-9}$. Though such
accuracy is quite impressive, it is far from the limitations of the machine
double precision, when one can expect relative errors as low as $2.3\cdot
10^{-16}$. This was a motivation for us to illustrate the numerical
performance of the new solution representations by showing that they can
produce significantly more accurate results for the same problem requiring
similar computation resources.

We used 2001 uniformly spaced points to represent all the functions involved
and computed the coefficients $\alpha_{n}$ and $\mu_{n}$ for $n\leq50$. The
coefficients were computed \textquotedblleft as is\textquotedblright, without
any special care about their behavior near $y=A$. Criteria based on the
equalities \eqref{alpha n verification}--\eqref{mu n verification2} showed
that optimal value for $N$ for the approximate solutions was $N=38$, which was
used to compute the first 100 eigenvalues. On Figure \ref{Ex1Fig1} we present
absolute and relative errors of the computed eigenvalues. The maximal absolute
error was $1.3\cdot10^{-11}$, and the maximal relative error was
$2.5\cdot10^{-15}$. The computation time on a computer equipped with Intel
i7-3770 CPU was about 0.25 seconds (a significant speed-up was achieved by
computing the spherical Bessel functions recurrently, see \cite{Barnett},
\cite{GillmanFiebig} and \cite{KNT} for details). \begin{figure}[tbh]
\centering
\includegraphics[bb=126 306 486 486, width=5in,height=2.5in]
{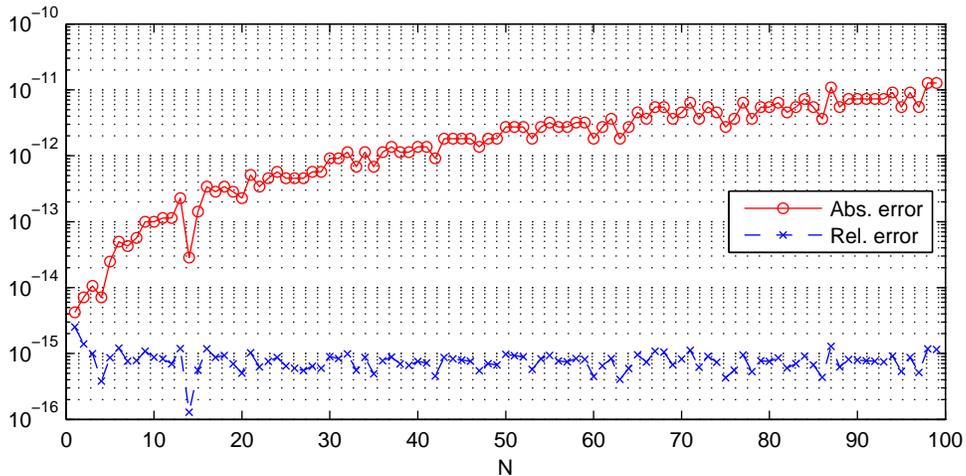}\caption{Absolute and relative errors of the first 100
eigenvalues for the spectral problem \eqref{Num Eq1}, \eqref{Num Eq1 BC}.}%
\label{Ex1Fig1}%
\end{figure}

\begin{figure}[tbh]
\centering
\includegraphics[bb=198 306 414 486, width=3in,height=2.5in]
{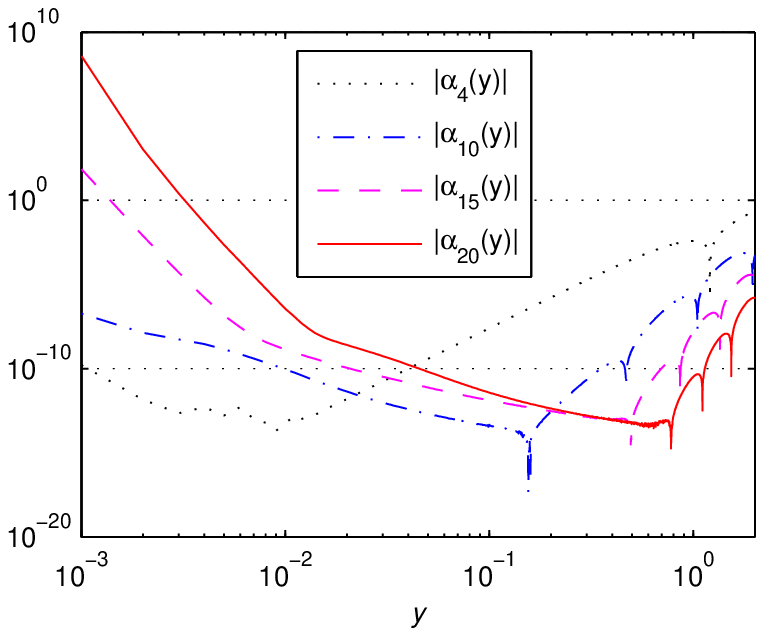}\quad\includegraphics[bb=198 306 414 486, width=3in,height=2.5in]
{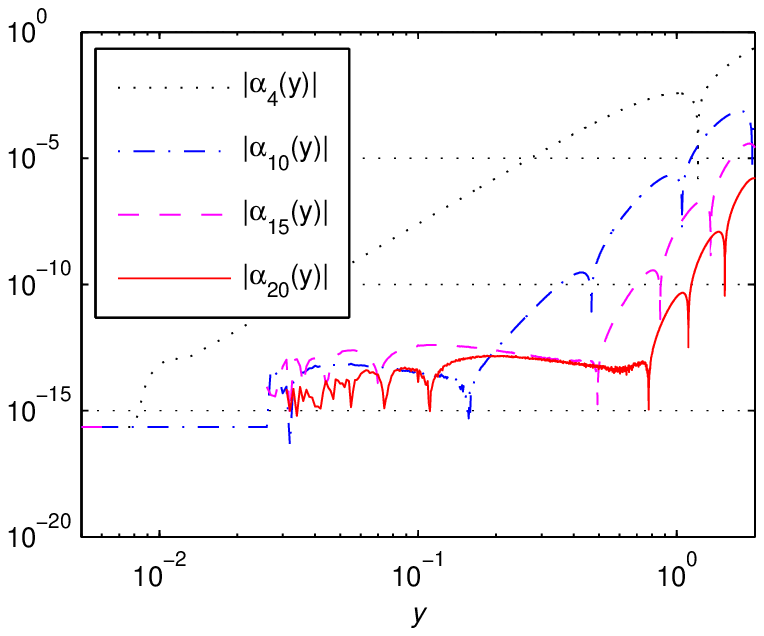}\caption{Absolute values of the coefficients $\alpha_{k}(y)$ for
$k=4,10,15,20$. On the left plot: calculated directly using the recurrent
formulas \eqref{Sigma and Upsilon}, \eqref{Sigma n}. Note the growth of the
values $|\alpha_{k}(y)|$ as $y\to0$, especially for large values of $k$. On
the right plot: calculated using the workaround proposed at the end of Section
\ref{SectNumAlg}. }%
\label{Ex1Fig2}%
\end{figure}

On Figure \ref{Ex1Fig2} we illustrate the difficulty in the computation of the
coefficients $\alpha_{n}$ and $\mu_{n}$ explained at the end of Section
\ref{SectNumAlg}. Direct implementation of the recurrent formulas
\eqref{Sigma and Upsilon}, \eqref{Sigma n} leads to large errors in the
coefficients in a neighborhood of $y=0$, while the coefficients computed using
the proposed workaround remain bounded in the same region.

In the last experiment we verified the accuracy of the approximate solutions.
One of the exact solutions of equation \eqref{Num Eq1} has the form
\[
u(y)=\exp\left(  y-\frac{iy^{2}\omega}{2}\right)  {}_{1}F_{1}\left(
\frac{1+i\omega}{4},\frac{1}{2},iy^{2}\omega\right)  ,
\]
where $\lambda=\omega^{2}$ and ${}_{1}F_{1}$ is the Kummer confluent
hypergeometric function. It satisfies the initial conditions $u(0)=u^{\prime
}(0)=1$. We calculated this solution numerically for three different values of
$\omega$, $52$, $105$ and $210$ (being close to 50th, 100th and 200th
eigenvalues respectively). $N=38$ was used for the approximate solutions. The
coefficients $\alpha_{n}$ and $\mu_{n}$ were computed in two different ways:
directly using the recurrent formulas and with a special care for small values
of $y$ (as was explained in Section \ref{SectNumAlg}). On Figure \ref{Ex1Fig3}
we present the absolute errors of the approximate solutions. We omitted the
case $\omega=52$ since the errors are indistinguishable on the plots. For
$\omega=105$ or $\omega=210$ one starts to observe additional errors for small
values of $y$ due to errors in the coefficients $\alpha_{n}$ and $\mu_{n}$.

\begin{figure}[tbh]
\centering
Absolute errors of the approximate solution for $\omega=105$.
\par
\includegraphics[bb=198 324 414 468, width=3in,height=2in]
{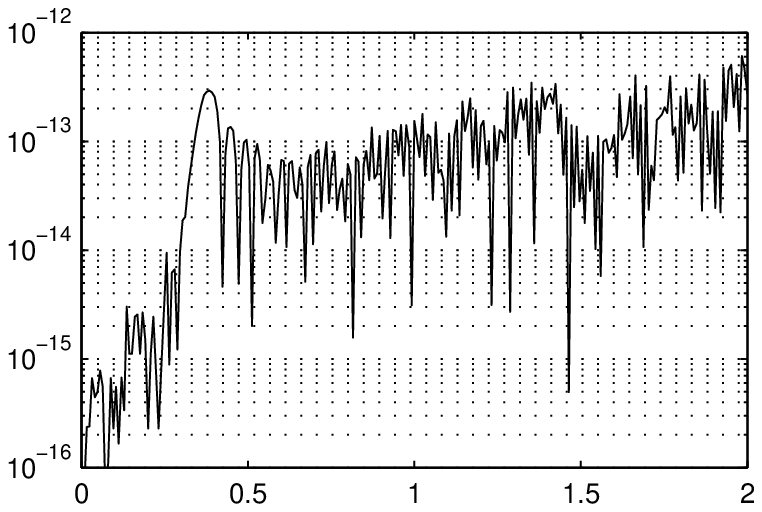}\quad\includegraphics[bb=198 324 414 468, width=3in,height=2in]
{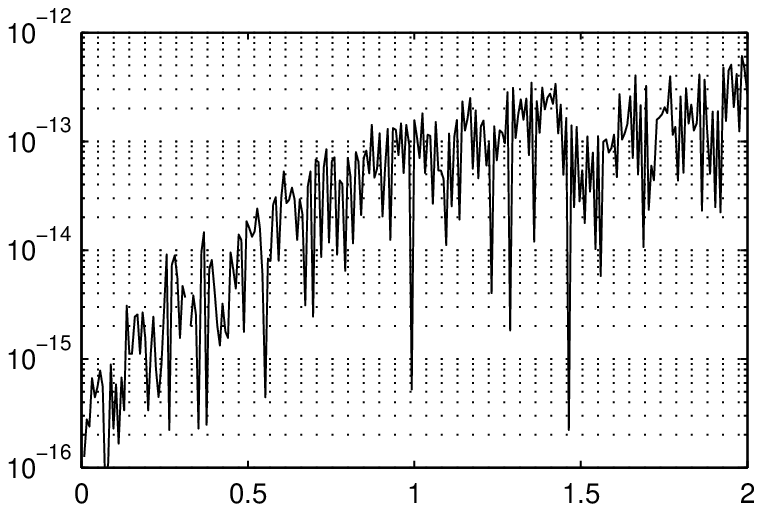}
\par
Absolute errors of the approximate solution for $\omega=210$.
\par
\includegraphics[bb=198 324 414 468, width=3in,height=2in]
{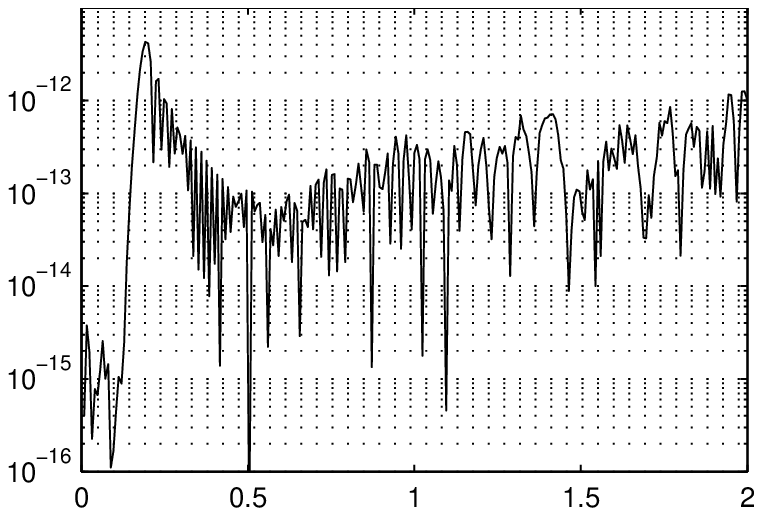}\quad\includegraphics[bb=198 324 414 468, width=3in,height=2in]
{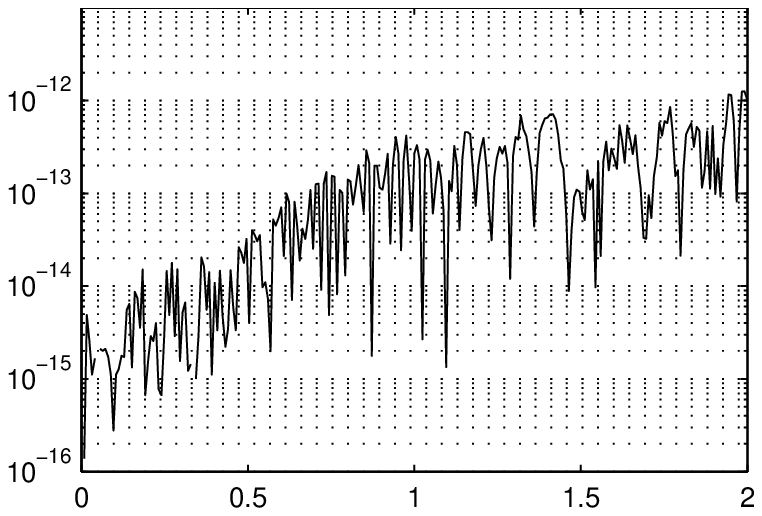}\caption{Absolute errors of the approximate solution $u_{N}%
(\omega, y)$. On the left: using the coefficients $\alpha_{n}(y)$ calculated
directly, without any special treatment for small values of $y$. On the right:
using the coefficients $\alpha_{n}(y)$ calculated applying the workaround
proposed at the end of Section \ref{SectNumAlg}. }%
\label{Ex1Fig3}%
\end{figure}

\appendix

\section{Error and decay rate estimates}

\label{Appendix Errors}

In this appendix some estimates for the functions $\varepsilon_{N}$ and
$\varepsilon_{1,N}$ from Theorems
\ref{Th Representation of solutions via Bessel} and
\ref{Th Derivatives Schrod} are presented. Also decay rate estimates and
bounds near $x=0$ for the coefficients $\beta_{n}$ and $\gamma_{n}$ are
obtained in dependence on the smoothness of the potential $Q$. The estimates
are from \cite{KNT} with some additional improvements obtained using ideas
from \cite{KTC}.

It is well known (see, e.g., \cite{LevitanInverse}, \cite{Marchenko},
\cite{Marchenko52}, \cite{Trimeche}) that the solutions $c(\omega,x)$ and
$s(\omega,x)$ of \eqref{Schr} satisfying the initial conditions
\eqref{init cs} can be represented as
\[
c(\omega,x)=\cos\omega x+\int_{-x}^{x}K(x,t)\cos\omega t\,dt
\]
and
\[
s(\omega,x)=\sin\omega x+\int_{-x}^{x}K(x,t)\sin\omega t\,dt,
\]
i.e., as images of the functions $\cos\omega t$ and $\sin\omega t$ under the
action of the so-called transmutation operator $T$. This operator is defined
on an arbitrary integrable function by the rule
\begin{equation}
Tu(x)=u(x)+\int_{-x}^{x}K(x,t)u(t)\,dt.\label{T def}
\end{equation}
Its integral kernel $K$ satisfies the equalities
\begin{equation}
K(x,x)=\frac{h}{2}+\frac{1}{2}\int_{0}^{x}Q(s)\,ds,\qquad K(x,-x)=\frac{h}%
{2},\label{K GoursatCond}
\end{equation}
and the derivative $\partial_{x}K=:K_{1}$ satisfies the equalities (see
\cite{KT FuncProp})
\begin{equation}
K_{1}(x,x)=\frac{Q(x)}{4}+\frac{h}{4}\int_{0}^{x}Q(s)\,ds+\frac{1}%
{8}\biggl(\int_{0}^{x}Q(s)\,ds\biggr)^{2},\quad K_{1}(x,-x)=\frac{1}%
{4}\biggl(Q(0)+h\int_{0}^{x}Q(s)\,ds\biggr).\label{K1 GoursatCond}%
\end{equation}

The integral kernel $K$ is one degree smoother than the potential $Q$. To be
more precise, let $Q\in W_{\infty}^{p}[0,b]$, $p\ge0$, i.e., the function $Q$
possesses $p$ derivatives, the last one belonging to $L_{\infty}(0,b)$. In
such case the integral kernel $K$ possesses $p+1$ derivative with respect to
each variable. In particular, for each $x>0$, $K(x,\cdot)\in W_{\infty}%
^{p+1}[-x,x]$ and the norms $\| \partial^{p+1}_{t} K(x,t)\|_{L_{\infty}%
(-x,x)}$ are bounded on $[0,b]$.

The following result shows that the coefficients $\beta_{n}$ and $\gamma_{n}$
defined by \eqref{beta direct definition} and \eqref{gamma n} are the
Fourier-Legendre coefficients of the integral kernel $K$ and its derivative
$K_{1}:=\partial_{x} K$, respectively.

\begin{theorem}
[\cite{KNT}]Let $Q\in L_{\infty}(0,b)$. The transmutation kernel $K$ and its
derivative $K_{1}$ have the form
\begin{equation}
\label{K and K1 Legendre}K(x,t) = \sum_{j=0}^{\infty}\frac{\beta_{j}(x)}{x}
P_{j}\left( \frac tx\right)  \qquad\text{and}\qquad K_{1}(x,t) = \sum
_{j=0}^{\infty}\frac{\gamma_{j}(x)}{x} P_{j}\left( \frac tx\right) ,
\end{equation}
here $P_{n}$ denotes the classical Legendre polynomials. The coefficients
$\beta_{n}$ and $\gamma_{n}$, $n\ge0$, can be recovered using the formulas
\begin{equation}
\label{betan gamman via K}\beta_{n}(x) = \frac{2n+1}{2}\int_{-x}^{x} K(x,t)
P_{n}\left( \frac tx\right) \,dt\qquad\text{and}\qquad\gamma_{n}(x) =
\frac{2n+1}{2}\int_{-x}^{x} K_{1}(x,t) P_{n}\left( \frac tx\right) \,dt.
\end{equation}

\end{theorem}

Denote by $K_{N}$ and $K_{1,N}$ partial sums of the series
\eqref{K and K1 Legendre}. Let $Q\in W_{\infty}^{p}[0,b]$, $p\ge0$. Define
\begin{equation}
\label{M and M1}M:=\sup_{0<x\le b}\| \partial_{t}^{p+1} K(x,\cdot
)\|_{L_{\infty}(-x,x)}\qquad\text{and}\qquad M_{1}:=\sup_{0<x\le b}\|
\partial_{t}^{p} K_{1}(x,\cdot)\|_{L_{\infty}(-x,x)}.
\end{equation}
As was mentioned earlier, both suprema exist and are finite numbers. The
following convergence rate estimates hold.

\begin{proposition}
Let $Q\in W_{\infty}^{p}[0,b]$, $p\geq0$. Then there exist constants $c_{p}$
and $d_{p}$ independent of $Q$ and $N$, such that for all $x>0$ the following
inequalities hold
\begin{equation}
\Vert K(x,\cdot)-K_{N}(x,\cdot)\Vert_{L_{2}(-x,x)}\leq\frac{c_{p}Mx^{p+3/2}%
}{N^{p+1}},\qquad N\geq p+2,\label{Estimate K minus K_N}%
\end{equation}
and
\begin{equation}
\Vert K_{1}(x,\cdot)-K_{1,N}(x,\cdot)\Vert_{L_{2}(-x,x)}\leq\frac{d_{p}%
M_{1}x^{p+1/2}}{N^{p}},\qquad N\geq p+1.\label{Estimate K1 minus K_1N}%
\end{equation}

\end{proposition}

\begin{proof}
Let $x>0$ be fixed. Consider functions $g(z):=K(x,xz)$ and $g_{N}%
(z):=K_{N}(x,xz)$ defined on $[-1,1]$. The function $g_{N}$ is a partial sum
of the Fourier-Legendre series for the function $g$, hence $g_{N}$ coincides
with the $N$-th order polynomial best $L_{2}$ approximation of the function
$g$. Since the integral kernel $K$ possesses $p+1$ derivatives with respect to
the second variable with the last derivative belonging to $L_{\infty}(-x,x)$,
we have that $g\in W_{\infty}^{p+1}[-1,1]\subset W_{2}^{p+1}[-1,1]$. Theorem
6.2 from \cite{DeVoreLorentz} states that there exists a universal constant
$\tilde{c}_{p}$ such that for every $N>p+1$
\[
\Vert g-g_{N}\Vert_{L_{2}(-1,1)}\leq\frac{\tilde{c}_{p}}{N^{p+1}}\omega\left(
g^{(p+1)},\frac{1}{N}\right)  _{2}\leq\frac{2\tilde{c}_{p}}{N^{p+1}}\Vert
g^{(p+1)}\Vert_{L_{2}(-1,1)},
\]
where $\omega$ is the modulus of continuity.

By the definition $g^{(p+1)}(z) = x^{p+1} \partial_{t}^{p+1}K(x,t)\big|_{t=xz}%
=:x^{p+1} K_{2}^{(p+1)}(x,xz)$, hence
\[%
\begin{split}
\|K(x,\cdot)-K_{N}(x,\cdot)\|_{L_{2}(-x,x)}  & = \sqrt{x}\|g-g_{N}%
\|_{L_{2}(-1,1)}\le\frac{2 \tilde c_{p} \sqrt x}{N^{p+1}} \|g^{(p+1)}%
\|_{L_{2}(-1,1)}\\
& =\frac{2 \tilde c_{p} x^{p+3/2}}{N^{p+1}} \|K_{2}^{(p+1)}(x,x\cdot
)\|_{L_{2}(-1,1)}\le\frac{2\sqrt2 \tilde c_{p} M x^{p+3/2}}{N^{p+1}},
\end{split}
\]
where we used that $\|K_{2}^{(p+1)}(x,x\cdot)\|_{L_{2}(-1,1)}\le\sqrt{2}M$ due
to \eqref{M and M1}.

The second estimate \eqref{Estimate K1 minus K_1N} can be obtained similarly.
\end{proof}

Estimates \eqref{Estimate K minus K_N} and \eqref{Estimate K1 minus K_1N}
provide upper bounds for the error functions $\varepsilon_{N}$ and
$\varepsilon_{1,N}$ from Theorems
\ref{Th Representation of solutions via Bessel} and
\ref{Th Derivatives Schrod}. Indeed, the approximate solutions $c_{N}$ and
$s_{N}$ are obtained by changing $K$ by $K_{N}$ in \eqref{T def}. Hence by the
Cauchy-Schwarz inequality we have
\begin{equation}%
\begin{split}
|c(\omega,x)-c_{N}(\omega,x)| &  \leq\biggl|\int_{-x}^{x}\bigl(K(x,t)-K_{N}%
(x,t)\bigr)\cos\omega t\,dt\biggr|\\
&  \leq\Vert K(x,\cdot)-K_{N}(x,\cdot)\Vert_{L_{2}(-x,x)}\cdot\biggl(\int
_{-x}^{x}|\cos\omega t|^{2}\,dt\biggr)^{1/2}\leq\frac{\sqrt{2}c_{p}Mx^{p+2}%
}{N^{p+1}}%
\end{split}
\label{c minus cn}%
\end{equation}
for any $\omega\in\mathbb{R}$. The estimates for the function $\varepsilon
_{1,N}$ follows similarly by using corresponding estimates for the derivative
$K_{1}$.

The following proposition provides estimates of the decay rate of the
coefficients $\beta_{k}$ and $\gamma_{k}$ and their behavior near $x=0$.

\begin{proposition}
\label{Prop Decay betas} Let $Q\in W_{\infty}^{p}[0,b]$, $p\geq0$. There exist
constants $C_{p}$ and $D_{p}$ (independent of $N$) such that
\begin{equation}
|\beta_{N}(x)|\leq\frac{C_{p}x^{p+2}}{(N-1)^{p+1/2}},\qquad N\geq
p+1\label{betan est}%
\end{equation}
and
\begin{equation}
|\gamma_{N}(x)|\leq\frac{D_{p}x^{p+1}}{(N-1)^{p-1/2}},\qquad N\geq
p.\label{gamman est}%
\end{equation}

\end{proposition}

\begin{proof}
We obtain using \eqref{betan gamman via K} and \eqref{Estimate K minus K_N},
the fact that the Legendre polynomial $P_{n}$ is orthogonal to any polynomial
of degree less than $N$ and the Cauchy-Schwarz inequality that
\[%
\begin{split}
|\beta_{N}(x)| &  =\frac{2N+1}{2}\biggl|\int_{-x}^{x}K(x,t)P_{N}\left(
\frac{t}{x}\right)  \,dt\biggr|\leq\frac{2N+1}{2}\int_{-x}^{x}\left\vert
(K(x,t)-K_{N-1}(x,t))P_{N}\left(  \frac{t}{x}\right)  \right\vert \,dt\\
&  \leq\frac{2N+1}{2}\Vert K(x,\cdot)-K_{N-1}(x,\cdot)\Vert_{L_{2}(-x,x)}%
\cdot\sqrt{\frac{2x}{2N+1}}\leq\sqrt{\frac{2N+1}{2}}\frac{c_{p}Mx^{p+2}%
}{(N-1)^{p+1}}.
\end{split}
\]
The proof of the second estimate \eqref{gamman est} is similar.
\end{proof}

The following bound for the behavior of the first coefficients near $x=0$ is valid.

\begin{corollary}
\label{Corr decay betas} Let $Q \in W_{\infty}^{p}[0,b]$, $p\ge0$. Then
\begin{align}
|\beta_{n}(x)| & \le c_{n} x^{n+1}, \qquad n\le p+1,\label{betan small n}\\
|\gamma_{n}(x)| & \le d_{n} x^{n+1},\qquad n\le p,\label{gamman small n}%
\end{align}
where $c_{n}$ and $d_{n}$ are some constants dependent on $Q$.
\end{corollary}

Sufficient conditions for the smoothness of the transformed potential $Q$ in
terms of the coefficients $p$, $q$ and $r$ can be easily obtained from \eqref{Q}.

\end{document}